\documentclass[english,twoside,11pt]{article}

\usepackage[lmargin=1in,rmargin=1in,tmargin=1in,bmargin=1in]{geometry}

\usepackage{float}
\usepackage{graphicx}
\usepackage{color}
\usepackage{tikz}
\usetikzlibrary{backgrounds}

\usepackage[margin=20pt]{caption}

\usepackage{amsmath}
\usepackage{amssymb}
\usepackage{amsthm}
\usepackage{mathrsfs}

\usepackage[english]{babel}
\usepackage{flafter}
\usepackage{array}
\usepackage{paralist}

\graphicspath{{.}{graphics/}}

\newtheorem{theorem}{Theorem}

\newtheorem*{corollary*}{Corollary}
\newtheorem{lemma}[theorem]{Lemma}

\newtheorem*{proposition*}{Proposition}

\newtheorem*{property*}{Property}

\theoremstyle{definition}
\newtheorem{definition}[theorem]{Definition}
\newtheorem*{definition*}{Definition}

\theoremstyle{remark}

\newtheorem*{remark*}{Remark}

\newtheorem*{example*}{Example}

\newcommand{\HH}{\mathcal{H}}
\newcommand{\VV}{\mathcal{V}}
\newcommand{\EE}{\mathcal{E}}

\newcommand{\OM}{\Omega}
\newcommand{\AUG}{\mathrm{AUG}}

\usepackage{framed}
\usepackage{algorithm}
\usepackage[noend]{algorithmic}

\input{mathdots}

\title{Non-Uniform Robust Network Design in Planar Graphs}

\author{
\begin{minipage}{\linewidth}
\begin{center}
David Adjiashvili\\[0.3\baselineskip]
\normalsize Institute for Operations Research, ETH Z\"urich\\
\normalsize addavid@ethz.ch
\end{center}
\end{minipage}
}

\date{April 17, 2015}

\begin{document}

\maketitle

\begin{abstract}
Robust optimization is concerned with constructing solutions that remain 
feasible also when a limited number of resources is removed from the solution. 
Most studies of robust combinatorial optimization to date made 
the assumption that every resource is equally vulnerable, and that the 
set of scenarios is implicitly given by a single budget constraint.
This paper studies a robustness model of a different kind.
We focus on \textbf{bulk-robustness}, a model recently introduced~\cite{bulk} 
for addressing the need to model non-uniform failure patterns in systems.

We significantly extend the techniques used in~\cite{bulk} to design
approximation algorithm for bulk-robust network design problems in
planar graphs. Our techniques use an augmentation framework, combined
with linear programming (LP) rounding that depends on a planar embedding of the input graph.
A connection to cut covering problems and the dominating set problem in
circle graphs is established. Our methods use few of the specifics 
of bulk-robust optimization, hence it is conceivable that they can 
be adapted to solve other robust network design problems.
\end{abstract}

\section{Introduction}

Robust optimization is concerned with finding solutions that perform well
in any one of a given set of scenarios. Many paradigms were proposed for
robust optimization is the last decades. Some models assume
uncertainty in the cost structure of the optimization problem. Such robust models
typically have as scenarios different cost structures for the resources in the
system, and ask to find a solution whose worst-case cost is as small as possible.
Another kind of robustness postulates uncertainty in the feasible set
of the optimization problem. Typically, in such models scenarios correspond to
different realizations of the feasible set. A minimum-cost solution is then sought 
that is feasible in any possible realization of the feasible set. 

This paper deals with the latter class of robust models, i.e. ones
that incorporate uncertainty in the feasible set. Concretely, 
we are interested in \textbf{robust network design problems}, that are generally
defined as follows. The input specifies a graph $G=(V,E)$, a set of 
\textbf{failure scenarios} $\OM$, consisting of subsets of the nodes and edges of $G$, 
and some connectivity requirement. The goal is to find a minimum-cost subgraph
of $G$ satisfying the connectivity requirement, even when the elements 
in any one single scenario are removed from the solution. Different problems
are obtained for different types of connectivity requirement and when different 
representations of the scenario set are assumed. 

Most existing models of robust network design assume \textbf{uniform scenario 
sets}. Given interdiction costs for the resources, and a bound $B\in \mathbb{Z}_{\geq 0}$, 
such models assume that the adversary can remove \textbf{any} subset of resources of
interdiction cost at most $B$. In fact, unit interdiction costs are almost always assumed.
While such uniform robust network design problems often enjoy good algorithms, 
they also often do not reflect realistically the uncertainty in the modeled system, which
feature highly non-uniform failure patterns.

In a recent paper, Adjiashvili, Stiller and Zenklusen~\cite{bulk} introduced 
a new model for robust network design called \textbf{bulk-robustness}, specifically 
designed to model such highly non-uniform failure patterns. In bulk-robust optimization
failure scenarios are given explicitly, as a list of subsets of the resources. These
subsets may be arbitrary, and in particular, they are allowed to vary in size. The
goal, as in robust network design problems, is to find a minimum-cost set of resources 
that contains a feasible solution, even when the resources in any one of the scenarios
are removed.

The authors justify the model by bringing many example from health care optimization,
computer systems, digitally controlled systems, military applications, financial systems
and more. For example, in computer systems, different components of the network rely 
on the different resources, such a databases, power sources etc. At down-times of such
resources, the components that can not operate properly are exactly those
that depend of the downed resource. While no uniform failure model can capture
such failure patterns, bulk-robutness seems to be a suitable choice.

In~\cite{bulk} the authors study a number of problems in the bulk-robust model, 
including the $s$-$t$ connection problem, and the spanning tree problem. In 
particular, the approximability of these problems is studied in general graphs.
The goal of this paper is to extend the existing tool set available for designing
approximation algorithm for robust network design in this model. In this paper
we focus on the important special case of \textbf{planar graphs}. We show 
a widely-applicable method for computing approximate solutions to bulk-robust
network design problem in planar graphs. 

\subsection{Results and methods}

For an integer $r\in \mathbb{Z}_{\geq 0}$ we let $[r]=\{1,\cdots,r\}$ and 
$[r]_0=\{0,1,\cdots,r\}$.
The \textbf{bulk-robust network design problem} is defined as follows.
Given an undirected graph $G = (V,E)$, a weight function 
$w:E\rightarrow \mathbb{Z}_{\geq 0}$, a connectivity requirement $\mathcal{C}$
and a set of $m$ \textit{scenarios} $F_1, \cdots, F_m$, each comprising a set of 
edges $F_i \subseteq E$, find a minimum-cost set of edges $S\subseteq E$, such that 
$(V, S\setminus F_i)$ satisfies $\mathcal{C}$ for every $i\in [m]$. When 
$\mathcal{C}$ is the requirement that two specific nodes $s,t\in V$ are to
be connected we obtain the \textbf{bulk-robust $s$-$t$ connection problem}. When
$\mathcal{C}$ is the requirement that all nodes are pair-wise connected, we obtain
the \textbf{bulk-robust spanning tree problem}. Other bulk-robust problems
such as \textbf{bulk-robust Steiner tree} and \textbf{bulk-robust survivable network design} are
obtained analogously, by choosing the appropriate $\mathcal{C}$.
We let $n = |V|$, and $k=\max_{i\in [m]} |F_i|$ denote the maximum size of a scenario. 
The parameter $k$ is called \textbf{the diameter} of the instance.
Adjiashvili et. al.~\cite{bulk} proved the following 
theorem.

\begin{theorem}[Adjiashvili et. al.~\cite{bulk}]\label{thm:oldapx}
 The bulk-robust $s$-$t$ connection problem admits a polynomial $13$-approximation
algorithm in the case $k = 2$. The bulk-robust spanning tree problem admits an 
$(\log n + \log m)$-approximation algorithm.
\end{theorem}

On the complexity side, the authors prove set cover hardness for all 
considered bulk-robust counterparts, implying a conditional $\log m$
lower bound on the approximation in general graphs. In terms of the
parameter $k$, the authors show that in general graphs a sub-exponential
approximation factor is likely not achievable for certain variants of the 
bulk-robust $s$-$t$ connection problem.

\subsubsection{Contribution}
Our goal is to prove a significant strengthening of Theorem~\ref{thm:oldapx}
for the special case where the input graph is planar. Concretely,
we prove the following theorem.

\begin{theorem}\label{thm:planar}
 The bulk-robust $s$-$t$ connection and the bulk-robust spanning tree problems 
admit polynomial $O(k^2)$-approximation algorithms on planar graphs.
\end{theorem}

The latter result implies constant-factor approximation algorithms
for the case of fixed $k$. In light of the results in~\cite{bulk} this is
qualitatively best possible. 
To complement our algorithmic result we also prove the following stronger
inapproximability result.

\begin{theorem}\label{thm:complexity}
 For some constant $c>0$ it is NP-hard to approximate the bulk-robust
$s$-$t$ connection problem within a factor of $ck$, even when the
input graphs are restricted to series-parallel graphs.
\end{theorem}

The expression $ck$ in the latter theorem can be replaced with the concrete
expression $\frac{1}{2}k-1+\frac{1}{2k} - \epsilon$.
Theorem~\ref{thm:complexity} suggests that the dependence of the approximation factor on $k$ is necessary.

For concreteness and clarity of the exposition we prove Theorem~\ref{thm:planar} 
for the $s$-$t$ connection problem. We discuss the necessary minor adaptation
needed for the spanning tree problem later. Furthermore, the methods we employ
use very little of the particularities of bulk-robust optimization, and are
thus likely to be adaptable to other robust problems on planar graphs. 

\subsubsection{Our methods}

Our algorihtm is a combination of combinatorial and LP-based techniques.
On the top level, our algorithm employs an augmentation framework, which
constructs a feasible solution by solving a sequence of relaxations of the
problem. The lowest level corresponds to a simple polynomial problem, while
the last level corresponds to the original instance. The idea of augmentation
is well known in the literature of network design (see e.g.~\cite{augmentation1,%
augmentation2}). 
We use here the variant of the augmentation framework defined for bulk-robust 
optimization in~\cite{bulk}. 

We solve each stage of the augmentation problem by considering a suitable
set cover problem, the analysis of which comprises the core technical
contribution of the paper. Using a combinatorial transformation that amounts
to finding certain shortest paths in the graph, we obtain a simpler covering
problem, which we call \textbf{the link covering problem}. The remainder of the 
algorithm relies on the analysis of the standard LP relaxation of the latter
problem. Using properties of planar graphs, we show that the obtained LP
has an integrality gap of $O(k)$, and that a solution of this quality can be obtained
in polynomial time. Our rounding procedure relies on a decomposition 
according to the planar embedding of the graph, and a connection to the
dominating set problem in circle graphs, for which we develop an LP-respecting
constant-factor approximation algorithm. 
The line of our proof follows that of the proof in~\cite{bulk}. Our main technical
contribution can hence be seen in the additional techniques developed to deal
with planar graphs. As we mentioned before, these new techniques seem more general
than the bulk-robust model, and are likely to be applicable to other network
design problems in planar graphs.

The proof of Theorem~\ref{thm:complexity} relies on a reduction form the minimum
vertex cover problem in $m$-uniform, $m$-partite hypergraphs. 

\subsubsection{Organization}

In the remainder of this section we review related work. In Section~\ref{sec:alg}
we present the algorithm for the bulk-robust $s$-$t$ connection problem, and
prove Theorem~\ref{thm:planar} for this case. The required modification 
for the bulk-robust spanning tree problem and possible extensions of our 
results are discussed in Section~\ref{sec:extensions}. 
The proof of Theorem~\ref{thm:complexity} is brought in Appendix~\ref{apx:complexity}.

\subsection{Related work}

For a comprehensive survey on general models
for robust optimization we refer the reader to the paper of
Bertsimas, Brown and Caramanis~\cite{BertsimasBrownCaramanis}.

Robustness discrete optimization with cost uncertainty was initially studied by Kouvelis and Yu~\cite{KouvelisYu} 
and Yu and Yang~\cite{YuYang}. These works mainly consider the min-max model, where the goal
is to find a solution that minimizes the worst-case cost according to the given set of cost functions.
See the paper of Aissi, Bazgan and Vanderpooten~\cite{Survey_AissiBazgenVanderpooten} for a survey.
%
%
A closely related class of multi-budgeted problem has 
received considerable attention recently
(see e.g.~\cite{ravi_1993_many,%
papadimitriou_2000_approximability,%
chekuri_2011_multibudgeted,%
grandoni2014new} and references therein).
%

%
An interesting class of problems with uncertainty in the feasible set was introduced by 
Dhamdhere, Goyal, Ravi and Singh~\cite{DhamdhereGoyalRaviSingh}. In this two-stage
models the feasibility condition is only fully revealed in the second stage. While resources
can be bought in both stages, they are cheaper in the first stage, in which only partial information
about the feasible set is available. This model was subsequently studied by several other authors 
(see~\cite{GolovinGoyalRavi,%
FeigeJainMahdianMirrokni,%
KhandekarKortsarzMirrokniSalavatipour}).
Different two-stage model was proposed in~\cite{AdjiashviliZenklusen1,%
adjiashvili2013online} 
for the shortest path problem.
Several other important network design problems are motivated by robust optimization.
Such problems include the minimum $k$-edge connected spanning subgraph 
problem~\cite{CheriyanThurimella,%
GabowGoemansTardosWilliamson}
and the survivable network design problem~\cite{Jain,%
survivable}.
Various other robust variants of classical combinatorial optimization problems were proposed. For a survey of these
results we refer the reader to the theses of Adjiashvili~\cite{thesis} and Olver~\cite{Olver}.

\section{Bulk-robust $s$-$t$ connection in planar graphs}\label{sec:alg}

In this section we are concerned with the \textbf{bulk-robust $s$-$t$ connection
problem}, which given an undirected graph $G = (V,E)$, a weight function 
$w:E\rightarrow \mathbb{Z}_{\geq 0}$, two terminals $s,t\in V$ and a set of $m$
scenarios $F_1, \cdots, F_m \subseteq E$, asks to find minimum-cost 
set of edges $S$, such that $S\setminus F_i$ contains an $s$-$t$ path for 
every $i\in [m]$.

The remainder of the section is organized as follows. First we explain the augmentation
framework in general. Then we define the set cover problem for the $i$-th augmentation
step and analyze its properties. Finally, we propose a LP-based approximation algorithm 
for the set cover problem.

\subsection{The augmentation framework}

Consider the following sequence of relaxations of the given instance of the
bulk-robust $s$-$t$ connection problem. For an integer $i\in [k]_0$ define
$\OM_i$ to be the collection of subsets of cardinality at most $i$ of 
the failure scenarios $F_1, \cdots, F_m$, i.e.
$$
\OM_i = \{F\subseteq E \,\mid\, \exists j\in [m] \,\, F\subseteq F_j \,\,\, \wedge \,\,\, |F|\leq i \}.
$$
Now, define the \textbf{$i$-th level relaxation $\mathrm{P}_i$} of our instance to be 
the instance where $\OM$ is replaced by $\OM_i$. Clearly $\mathrm{P}_0$ is 
simply the shortest path problem, as $\OM_0 = \{\emptyset\}$, and $\mathrm{P}_k$
is the original instance. Furthermore, we indeed obtained a sequence of relaxations, 
as any feasible solution for $\mathrm{P}_i$ is feasible for $\mathrm{P}_j$ if $i\geq j$.

The augmentation framework constructs the solution for the given instance by
iteratively adding additional edges to the solution. The solution $X_{i-1}$ 
obtained until the beginning of the $i$-th augmentation step is feasible for 
$\mathrm{P}_{i-1}$. The \textbf{$i$-th augmentation problem} is to augment $X_{i-1}$
with additional edges $A_i$ of minimum cost so that $X_{i-1} \cup A_i$ is feasible for $\mathrm{P}_i$.
We denote by $\AUG_i$ the optimal value the $i$-th augmentation problem.

\subsection{The $i$-th augmentation problem}

In the first iteration, the problem $\mathrm{P}_0$ becomes the shortest $s$-$t$ path 
problem, and is solved in polynomial time by any shortest path algorithm.
We denote by $X_{i-1}\subseteq E$ the set of edges presented to the $i$-augmentation
problem, and let $G_{i-1} = (V, X_{i-1})$.

Consider the $i$-th augmentation problem for some $i\geq 1$. Since $X_{i-1}$
is a feasible solution of $\mathrm{P}_{i-1}$, we know that any scenario 
$\OM_{j}$ for $j<i$ does not disconnect $s$ from $t$ in $G_{i-1}$. The same
may hold true for some scenarios in $\OM_i$. If this holds for all scenarios in 
$\OM_i$, then $X_{i-1}$ is already feasible for $\mathrm{P}_i$, and we can set
$X_i = X_{i-1}$. In the other case, some scenarios in $\OM_i$ are still relevant,
i.e. they disconnect $s$ from $t$ in $G_{i-1}$. 
We abuse notation and let $\OM_i$ denote this set of relevant scenarios.

Let us formulate the $i$-th augmentation problem as a set cover problem.
To this end we let $E_i = E \setminus X_{i-1}$ denote the set of edges
not yet chosen to be included in the solution. Let $\bar V = V[X_{i-1}]$
be the set of nodes incident to $X_{i-1}$. Let us define the following useful
notion of links.

\begin{definition}\label{def:link}
 Let $u,v\in \bar V$ be distinct nodes. Define the \textbf{$u$-$v$ link} $L_{u,v}$
to be any shortest $u$-$v$ path in $(V,E_i)$. Let $\ell_{u,v} = w(L_{u,v})$ denote
the length of this path.
\end{definition}

Consider any optimal solution $A^*$ to the $i$-th augmentation problem. It
is easy to see that $A^*$ is acyclic, i.e. it forms a forest in $(V, E_i)$. 
Instead of looking for forests, however, we would like to restrict our search to
collections of links.

The advantage of using links is twofold. On the one hand, it is possible to
compute all links using a shortest path algorithm in polynomial time. On the
other hand, using links will allows us to decompose the augmentation problem
in a later stage. Let us define the notion of covering with links next.

\begin{definition}\label{def:crossing}
 A link $L_{u,v}$ is said to \textbf{cover $F\in \OM_i$} if 
its endpoints $u$ and $v$ lie on different sides of the cut formed by 
$F$. 
\end{definition}

It is easy to see that a union of links forms a feasible solution to
the augmentation problem if and only if for every set $F\in \OM_i$, at least one
of the links in the union covers $F$. 
This formulation naturally gives rise to our desired set cover problem, defined next.

\begin{definition}\label{def:setcover}
 The \textbf{$i$-th link covering problem} asks to find a collection 
of links of minimum total cost, covering every scenario $F\in \OM_i$.
\end{definition}

We also know that feasible solutions to the $i$-th links covering problem
correspond to feasible solutions of the $i$-th augmentation problem with the
same objective function value, or better. 
The following lemma from~\cite{bulk} states
that any feasible solution to the $i$-th augmentation problem corresponds
to a feasible solution to the $i$-th links covering problem of at most twice
the cost, thus by solving the link covering problem we lose at most a factor of $2$.

\begin{lemma}[Adjiashvili et. al.~\cite{bulk}]\label{lem:paths}
There exists a collection $Q_1, \cdots, Q_r \subseteq E_i$ of paths,
such that for each $F\in \OM_i$, the collection contains at least
one path covering $F$ and $\sum_{j=1}^{r}{w(Q_j)} \leq 2\mathrm{AUG}_i$.
\end{lemma}

Before proposing an approximation algorithm for the link covering problem let
us make the following additional assumption. We assume that \textbf{every} 
edge $e\in X_{i-1}$ appears in at least one scenario from $\OM_i$. This
assumption does not compromise generality, as any edge not satisfying the 
latter condition can be safely contracted for the solution of the $i$-th
augmentation problem.

\subsection{Approximating the link covering problem}

We focus next on approximating the $i$-th link covering problem. For simplicity
we drop the index $i$ from our notation in this section and use
$X, \OM$ and $\mathrm{P}$ for $X_{i-1}, \OM_i$ and $\mathrm{P}_i$, respectively.
The case $i=1$ is particularly simple and is treated as follows. In the
case $i=1$ the set $X$ simply corresponds to an $s$-$t$ path. This
path can be seen as a line and links can be seen as intervals on this line.
Scenarios $F\in \OM$ are singletons corresponding to edges on this path, and
are interpreted as points on the line. The link covering problem now becomes
an interval covering problem that can be solved exactly in polynomial
time using various algorithms.

In the case $i\geq 2$, which we henceforth assume, the situation is much more complex. 
Consider next the following standard linear programming relaxation of the link covering problem.
We include a variable $x_{u,v}\in [0,1]$ for each link $L_{u,v}$,
where $x_{u,v} = 1$ is interpreted as including the link $L_{u,v}$.
Furthermore, we denote by $cover(F)$ all pairs $\{u,v\} \subseteq \bar V\times \bar V$ 
such that the link $L_{u,v}$ covers $F$. 

\begin{equation*}
\min \left\{ \ell(x) : x_{u,v} \geq 0 \;\forall \{u,v\}\in \bar V\times \bar V, \quad \sum\limits_{\{u,v\}\in cover(F)} x_{u,v} \geq 1 
\quad \forall F \in \OM \right\}
\end{equation*}

It is well-known that in general, the latter LP has an integrality gap as large
as $\log N$, where $N$ is the size of the ground set of the set cover problem.
Our goal here is to show that in the case of the link covering problem and
when the input graph is required to be \textbf{planar}, 
a stronger bound can be proved. Concretely, we will show that a fractional solution $x^*$
to the LP can be rounded in polynomial time to an integral solution with cost at most
$8i\ell(x^*)$, thus also proving a bound of $8i$ on the integrality gap.

\subsubsection{Solving the LP}

Before we turn to our rounding algorithm, let us discuss the problem of solving
the latter LP. Clearly, if $k$ is a fixed constant, the size of the LP
is polynomial, and any polynomial time LP algorithm can be used. In the other
case, when the diameter $k$ is not bounded by a constant, the sets $\OM$ might
have exponential size, as they potentially contain all subsets of cardinality $i\leq k$
of sets of cardinality $k$. It is however not difficult to design a polynomial-time 
separation procedure for the latter LP as follows. Given a fractional vector $x$,
we can check if it is feasible for the LP by checking for every one of the polynomially many
failure scenarios $F_1, \cdots, F_m$, if it contains a subset $F$ of size $i$
that is both an $s$-$t$ cut in $\bar G = (V,X)$, and 
$$
\sum_{\{u,v\} \in \bar V\times \bar V \,\,\, \text{covers} \,\, F} x_{u,v} < 1.
$$
Let us call a set $F$ of the latter type \textbf{violating}.
This can be achieved as follows. Let $F_j$ be the scenario from the family of input
scenarios that we would like to test. Let $H = (\bar V,Y)$ be the graph obtained from $\bar G$
by adding the direct edge $\{u,v\}$ for every pair of distinct nodes $u,v \in \bar V$. The new
edge $\{u,v\}$ represents the link $L_{u,v}$.
Define an edge capacity vector $c: Y\rightarrow \mathbb{R}_{\geq 0}$ 
on the new edge set $Y$ setting $c_j(e) = 1$ if $e\in F_j$, $c_j(e) = \infty$ if $e\in X\setminus F_j$
and $c_j(e) = x_e$ if $e\in Y\setminus X$. It is now easy to verify that a violating 
set $F \subseteq F_j$ exists if and only if the capacity of the minimum $s$-$t$ cut in $H$ with 
capacity vector $c$ is strictly bellow $i+1$. Furthermore, if such a cut exists, the 
set $F$ can be chosen to be all edges of $F_j$ crossing the minimum cut. 
Polynomiality of the latter transformation and the minimum $s$-$t$ cut problem
now imply that the Ellipsoid algorithm can be used to solve the LP in polynomial time.


\subsubsection{Rounding the LP}

Let $x^*$ denote an optimal solution to our LP. We describe our rounding procedure
next.

Our rounding technique heavily exploits the planarity of the input graph $G$.
Let us henceforth assume that $G$ is presented with a planar embedding $\Gamma$. 
Such an embedding can be computed in polynomial time. 
We let $\psi^1,\cdots, \psi^{\bar q}$ denote the faces of the embedding of $\bar G$, induced
by the embedding of $G$.

\begin{definition}\label{def:planarproperties}

We say that link $L_{u,v}$ is of \textbf{type $j$} if it connects two nodes
$u, v$ on $\psi^j$, and if $L_{u,v}$ is completely contained in the face $\psi^j$.
We call a link \textbf{typed}
if it is of type $j$ for some $j\in [\bar q]$. Links that are not typed are called
\textbf{untyped}.
\end{definition}

For what follows it will be convenient to assume that $x^*$ is \textbf{clean}, i.e.
that $x^*_{u,v} = 0$ holds for every untyped link $L_{u,v}$. 
This assumption does not compromise
generality, as we state in the following lemma.

\begin{lemma}\label{lem:onlytyped}
 Restricting the solutions of the LP to satisfy $x_{u,v} = 0$ for every untyped link 
$L_{u,v}$ does not change the optimal value of the LP.
\end{lemma}

\begin{proof}
Assume that $x^*$ is an optimal solution to the LP with minimum possible 
weight assigned to untyped links
$$
\sum_{\{u,v\}\in \bar V\times\bar V \,\,\, \text{untyped}} x^*_{u,v}.
$$
Assume towards contradiction that $x^*_{u,v}>0$ holds for some untyped link $L_{u,v}$.
Since $L_{u,v}$ is untyped, it forms a shortest path between the nodes
$u$ and $v$, composed of edges contained in several faces of $\bar G$.
Let $u = v_1, \cdots, v_{p} = v$ be nodes on $L_{u,v}$ with the following properties.
\begin{itemize}
 \item The nodes appear in this order on $L_{u,v}$, when it is traversed from $u$ to $v$.
 \item For every $i\in [p-1]$, it holds that $L_{v_i, v_{i+1}}$ is a typed
link, i.e. it holds that $v_i, v_{i+1}\in \bar V$ and the sub-path of $L_{u,v}$ 
between $v_i$ and $v_{i+1}$ is completely contained in some face $\psi^{j_i}$.
\end{itemize}
Now, consider the LP solution $y$ where 
\begin{itemize}
 \item $y_{u,v} = 0$,
 \item $y_{v_i, v_{i+1}} = \min \{1, x^*_{v_i, v_{i+1}} + x^*_{u,v}\}$ for every $i\in [p-1]$, and
 \item $y_{w,z} = x^*_{w,z}$ everywhere else.
\end{itemize}

Since all links are shortest paths we have $w(L_{u,v}) = \sum_{i\in [p-1]} w(L_{v_i, v_{i+1}})$,
and thus $\ell_{u,v} = \sum_{i\in [p-1]} \ell_{v_i, v_{i+1}}$. This implies that $\ell(y)\leq \ell(x^*)$.

The new solution is also a feasible LP solution. To see this we only need to verify
that for every $F\in \OM$, the constraint 
$$
\sum\limits_{\{z,w\}\in cover(F)} y_{z,w} \geq 1
$$
holds. If $L_{u,v}$ does not cover $F$, this is obvious from feasibility of $x^*$, 
since $y_{z,w} \geq x^*_{z,w}$ for all links except $L_{u,v}$. 

In the remaining case $L_{u,v}$ covers $F$. Now, since the union of links 
$\cup_{i\in [p-1]} L_{v_i, v_{i+1}}$ contains a $u$-$v$ path, clearly
at least one of these links, say $L_{v_{i^*}, v_{i^*+1}}$, also covers $F$.
If $y_{v_{i^*}, v_{i^*+1}} = 1$ we are clearly done. In the other case
$$
y_{v_{i^*}, v_{i^*+1}} = x^*_{v_{i^*}, v_{i^*+1}} + x^*_{u,v},$$
and thus what is lost by reducing $x^*_{u,v}$ is compensated by increasing 
$x^*_{v_{i^*}, v_{i^*+1}}$, and the constraint is also satisfied.

Finally, we obtained a new optimal solution $y$ with a lower weight assigned
to untyped links, as all the links of the form $L_{v_i, v_{i+1}}$ are typed links,
and the link $L_{u,v}$ is untyped. 
This contradicts the choice of $x^*$. 
\end{proof}


A set of links $S$ is \textbf{clean} if it only contains typed links.
The following lemma proves certain useful connections between the planar embedding
of $G$ and the link covering problem, which we later use to round the LP solution.
We say that an edge is \textbf{on the boundary of a face} if both of its endpoint lie 
on the face.

\begin{lemma}\label{lem:coverproperties}
 Let $F\in \OM$ be some failure scenario and let $\psi \in \{\psi^1, \cdots, \psi^{\bar q}\}$
be some face. Then, if $i\geq 2$, the number of edges of $F$ that lie on the boundary of $\psi$ is 
either zero or two. Furthermore, the number of faces $\{\psi^1, \cdots, \psi^{\bar q}\}$ that
contain two edges of $F$ on their boundary is exactly $i$.
\end{lemma}

\begin{proof}
 Since $F\in \OM$ we know that $F$ is an $s$-$t$ cut in $\bar G$. Observe that
$(\bar V, X\setminus F)$ contains exactly two connected components, one $C^s(F)$ containing $s$
and one $C^t(F)$ containing $t$. This holds since, by definition of $\OM$ and the
augmentation problem, the set $X$ is feasible for $\mathrm{P}_{i-1}$, and thus any subset of $F$ 
is \textbf{not} an $s$-$t$ cut in $\bar G$.

This implies that all edges in $F$ can be directed unambiguously from the node in $C^s(F)$
to the node in $C^t(F)$. Now consider any edge $e\in F$ and  any face 
$\psi \in \{\psi^1, \cdots, \psi^{\bar q}\}$ which contains $e$ on its boundary.
Since $\psi$ corresponds to a cycle in $\bar G$, the number of edges of $F$ on its 
boundary cannot be odd, as an odd number of such edges would imply the existence
of path in $\bar G$ connecting $C^s(F)$ to $C^t(F)$, and containing no edge of $F$.

Next we prove that this number must be two, i.e. that $\psi$ contains exactly one
more edge of $F$. Assume towards contradiction that there are at least four such edges.
By traversing the cycle in $\bar G$, forming the face $\psi$, the cut defined by $F$ 
is crossed every time an edge of $F$ is crossed. In particular, there are 
some four nodes $u_1, v_1, u_2, v_2$ appearing in this order on the face, and such
that $u_1, u_2$ belong to $C^s(F)$ and $v_1, v_2$ belong to $C^t(F)$. Let $Q$ and $R$ be
a $u_1$-$u_2$ path in $C^s(F)$ and a $v_1$-$v_2$ path in $C^t(F)$, respectively. 
Since $\psi$ is a face, the embedding of both $Q$ and $R$ is disjoint from the
interior of $\psi$. Now, since $Q$ and $R$ form continuous curves in the plane, and
are connected to alternating nodes on the boundary of a face, they
must intersect at some point, contradicting the fact that $C^s(F)$ and $C^t(F)$ are
different connected components in $(\bar V, X\setminus F)$. Figure~\ref{fig:lem3}
illustrates this argument.

Finally, since every edge $e\in F$ belongs to the boundary of exactly two 
faces in $\{\psi^1, \cdots, \psi^{\bar q}\}$, and since every face containing
some edge of $F$ on the boundary contains exactly two such edges, we 
conclude that there are exactly $|F| = i$ faces containing some edge of $F$
on the boundary. In the first assertion we assumed there are no cut edges
in $\bar G$. For $i\geq 2$ this can be assume without loss of generality, as
cut edges are either contracted in the pre-processing stage before the augmentation
step, or, they are redundant, and can be removed from $X$.

\end{proof}

\begin{figure}
\begin{center}
\begin{tikzpicture}[scale=1.4]
\draw [black, very thick, fill=blue!10] plot [smooth cycle] coordinates {(0,0) (-0.3,0.05) (-0.5,0.2) (-0.6, 0.5) (-0.5,0.9) (-0.3,1.1) 
(0.6,1.3) (1,1.2) (1.2,1.0) (1.4, 0.5) (1.1,0.1) (0.9,0)};

\node (a1) at (0.4,0.6) [circle,draw=blue!10,fill=blue!10,thick,inner sep=1pt,minimum size=1mm] {\Large $\psi$};

\draw [ultra thick,red] (-0.3,0.05) to[out=150,in=-60] (-0.5,0.2);
\node (a1) at (-0.3,0.05) [circle,draw=black!100,fill=black!100,thick,inner sep=1pt,minimum size=1mm] {};
\node (a2) at (-0.5,0.2) [circle,draw=black!100,fill=black!100,thick,inner sep=1pt,minimum size=1mm] {};

\draw [ultra thick,red] (-0.5,0.9) to[out=70,in=200] (-0.3,1.1); 
\node (a1) at (-0.5,0.9) [circle,draw=black!100,fill=black!100,thick,inner sep=1pt,minimum size=1mm] {};
\node (a2) at (-0.3,1.1) [circle,draw=black!100,fill=black!100,thick,inner sep=1pt,minimum size=1mm] {};

\draw [ultra thick,red] (1,1.2) to[out=-30,in=120] (1.2,1.0);
\node (a1) at (1,1.2) [circle,draw=black!100,fill=black!100,thick,inner sep=1pt,minimum size=1mm] {};
\node (a2) at (1.2,1.0) [circle,draw=black!100,fill=black!100,thick,inner sep=1pt,minimum size=1mm] {};

\draw [ultra thick,red] (1.1,0.1) to[out=230,in=0] (0.9,0);
\node (a1) at (1.1,0.1) [circle,draw=black!100,fill=black!100,thick,inner sep=1pt,minimum size=1mm] {};
\node (a2) at (0.9,0) [circle,draw=black!100,fill=black!100,thick,inner sep=1pt,minimum size=1mm] {};


\node (b1) at (0,0) [circle,draw=black!100,fill=black!100,thick,inner sep=1pt,minimum size=1mm] {};
\node (b2) at (-0.6,0.5) [circle,draw=black!100,fill=black!100,thick,inner sep=1pt,minimum size=1mm] {};
\node (b3) at (0.6,1.3) [circle,draw=black!100,fill=black!100,thick,inner sep=1pt,minimum size=1mm] {};
\node (b4) at (1.4,0.5) [circle,draw=black!100,fill=black!100,thick,inner sep=1pt,minimum size=1mm] {};

\node (a1) at (-1.1,0) [circle,draw=blue!0,fill=blue!0,thick,inner sep=1pt,minimum size=1mm] {\large $Q$};
\node (a1) at (-1.45, 1.3) [circle,draw=blue!0,fill=blue!0,thick,inner sep=1pt,minimum size=1mm] {\large $R$};

\draw [black, thick] plot [smooth] coordinates {(0,0) (-0.7,-0.2) (-0.9,0.4) (-0.9, 1.1) (-0.1,1.8) (0.6,1.3)};
\draw [black, thick] plot [smooth] coordinates {(-0.6,0.5) (-1.3, 0.7) (-0.9, 1.7) (0.7, 1.9) (1.8, 0.8) (1.4,0.5)};

\node (b4) at (-0.92,0.52) [circle,draw=blue!100,fill=blue!100,thick,inner sep=1pt,minimum size=2mm] {};
\end{tikzpicture}
\end{center}
\caption{The situation in the proof of Lemma~\ref{lem:coverproperties}.}\label{fig:lem3}
\end{figure}
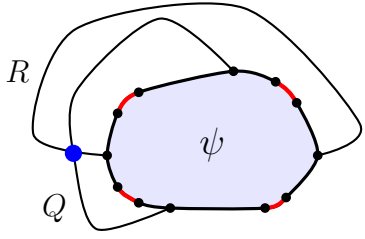

For simplicity we say that a \textbf{scenario $F\in \OM$ is contained in a 
face $\psi \in \{\psi^1,\cdots, \psi^{\bar q}\}$},
if two edges of $F$ lie on the boundary of $\psi$.
Lemma~\ref{lem:coverproperties} implies that a clean set $S$ of links is 
feasible if and only if for every $F\in \OM$, there exists a face $\psi^j$
containing $F$, and a link $L_{u,v} \in S$ of
type $j$ with $u$ and $v$ on different sides of the cut defined by $F$.
With this criterion we are ready to prove the main lemma of this section. 

\begin{lemma}\label{lem:rounding}
 Let $x$ be a clean feasible solution to the LP. Then, there exists a feasible set $S$ 
of links with total cost at most $8i\ell(x)$.
\end{lemma}

\begin{proof}
 We construct the desired set of links in two steps. First, we partition
the set of scenarios $\OM$ into $\bar q$ parts, one for each face of $\bar G$.
In the second stage, we process the faces of $\bar G$ one by one, and for each face
we use the part of the LP solution $x$ corresponding to the face to construct
a set of links that cover the scenarios assigned to that face.

Consider any scenario $F\in \OM$. Since $x$ is feasible we have 
$$
\sum\limits_{\{u,v\}\in cover(F)} x_{u,v} \geq 1.
$$
Let $\psi^{p_1}, \cdots, \psi^{p_i} \in \{\psi^1, \cdots, \psi^{\bar q}\}$ be
the set of faces that contain $F$. According
to Lemma~\ref{lem:coverproperties}, there are exactly $i$ such faces.
Now, since $x$ is clean, the latter sum 
can be decomposed as follows.
$$
\sum\limits_{\{u,v\}\in cover(F)} x_{u,v} = \sum_{j=1}^i \,\,
\sum\limits_{\substack{\{u,v\}\in cover(F) \\ L_{u,v} \,\, \text{type} \,\, p_j}} x_{u,v}
$$
Let us denote the second sum on the right hand side by $\sigma^j[F]$, i.e. let
$$
\sigma^j[F] = \sum\limits_{\substack{\{u,v\}\in cover(F) \\ L_{u,v} \,\, \text{type} \,\, p_j}} x_{u,v}.
$$
Now since $\sum_{j=1}^i \sigma^j[F] \geq 1$, there exists at least one index $j\in [i]$ such that
$\sigma^j[F] \geq \frac{1}{i}$. We let $j[F]\in [i]$ be one such index. If several indices $j\in [i]$
satisfy the latter condition, one is chosen arbitrarily. Note that the index $j[F]$ is chosen
is such a way that in the LP solution $x$, the total weight of links of type $p^{j[F]}$ that 
cover $F$ is at least $\frac{1}{i}$.

We are now ready to define the partition of $\OM$ into $\bar q$ parts, corresponding to the
$\bar q$ faces of $\bar G$. For $j\in [\bar q]$ we let
$$
\OM^{(j)} = \left\{ F\in \OM \,\mid\, p^{j[F]} = j\right\}.
$$
Clearly, $\OM = \cup_{j\in [\bar q]} \OM^{(j)}$ is a partition of $\OM$. To conclude the first
stage of the our procedure, it remains to define a corresponding decomposition 
$
x = \sum_{j\in [\bar q]} x^{(j)}
$
of the LP solution $x$. The vector $x^{(j)}$ is defined by setting $x^{(j)}_{u,v} = x_{u,v}$
if $L_{u,v}$ is of type $j$, and $x^{(j)}_{u,v} = 0$ otherwise. This concludes the first
step of the rounding procedure.

In the second step we construct for every $j\in [\bar q]$, a set of links $S^{(j)}$
of total cost $8i\ell(x^{(j)})$ that covers all scenarios in $\OM^{(j)}$. By doing so
we clearly conclude the proof of the lemma, since by taking $\cup_{j\in [\bar q]} S^{(j)}$
we obtain a feasible solution with total cost of at most
$\sum_{j\in [\bar q]} 8 i \ell(x^{(j)}) = 8 i \ell(x)$, as desired.

It remains to show how a single set $S^{(j)}$ can be constructed. Our plan
is the following. First, we observe that, by construction, $i x^{(j)}$ is an 
LP solution that fractionally covers all scenarios in $\OM^{(j)}$. Then, 
we observe that the link covering problem restricted to links of type $j$,
and to scenarios in $\OM^{(j)}$ essentially becomes a variant of the \textbf{dominating set
problem on circle graphs.} We explain the required transformation next,
and conclude by proving that the integrality gap of the standard LP
relaxation for the latter problem is constant, and that the corresponding
integral solution can be found in polynomial time.



Recall that a \textbf{circle graph} is an intersection graph of the
set of chords in a circle. The dominating set problem in circle graphs hence
corresponds to finding a minimum-cost collection of chords that 
intersect every chord of the graph. We are interested in a variant of this
problem, where chords are partitioned into two groups called \textbf{demand chords}
and \textbf{covering chords}, and the goal is to find a minimum-cost set of covering 
chords that dominates all the demand chords. We call this problem the \textbf{restricted 
dominated set problem in circle graphs}.

The link covering problem restricted to
a the face $\psi$ can now be seen as a dominating set problem on circle graphs 
as follows. Let $v_0, v_1, \cdots, v_d = v_0$ be the nodes on the boundary
of $\psi$. We subdivide each edge $\{v_j, v_{j+1}\}$ for $j\in [d]$ by adding
the node $w_j$. This new cycle corresponds to the circle of the circle graph
we construct. Let us define the chords of the graph, and their corresponding weights,
next. For every scenario $F$ contained in $\psi$ we add the demand chord $\alpha_F$ connecting $w_{j_1}$
to $w_{j_2}$, where $\{u_{j_1}, u_{j_1+1}\}\in F$ and $\{u_{j_2}, u_{j_2+1}\} \in F$
(recall from Lemma~\ref{lem:coverproperties} that there are exactly two such edges).
Next, for every link of the form $L_{v_l, v_r}$, we add the covering chord 
$\beta_{v_l, v_r}$ connecting $v_l$ with $v_j$. The cost of this chord is set to 
$\ell_{v_l, v_r}$, i.e. we set $c(\beta_{v_l, v_r}) = \ell_{v_l, v_r}$. 
This concludes the transformation.

To see that the latter problem indeed models the desired link covering
problem it suffices to make the following simple observation. Sets of
chords corresponding to links that form a restricted dominating set in the circle
graph are in one-to-one correspondence with sets of links that cover all
scenarios, with identical costs. This is true, since a link 
$L_{v_l, v_r}$ covers a scenario $F$ if and only if the chords $\alpha_{F}$
and $\beta_{v_l, v_r}$ intersect.

We can now naturally interpret the solution $y^{(j)} = ix^{(j)}$ as a feasible 
fractional solution to the standard LP relaxation of the restricted dominating set
problem on the obtained circle graph. In the following claim we show that
the integrality gap of the latter LP is constant.
The proof of the claim uses a connection
to a special case of the \textbf{axes-parallel rectangle covering problem}, 
for which Bansal and Pruhs~\cite{bansal2010geometry} provided an LP-respecting
$2$-approximation with the natural LP.
This concludes the proof of the lemma.

\vspace{2mm}
\noindent\textit{Claim 1.}
 The integrality gap of the standard LP relaxation of the restricted 
dominating set problem on circle graphs is bounded by $8$.

\begin{proof}
Let $H = (V^d\cup V^c, E)$ be the given circle graph with $V^d$ and
$V^c$ corresponding to the demand chords and the covering chords, respectively.
Let $g: V^c \rightarrow \mathbb{Z}_{\geq 0}$ denote the cost function for the
covering chords. Let $p_0, \cdots, p_m = p_0$ be all the points on the circle to which 
chords are connected, in the order that they appear when the circle is traversed in some
arbitrary direction. For a chord $\alpha\in V^d\cup V^c$ we write $\alpha = (p_l,p_r)$ 
with $l \leq r$ to indicate the endpoints of the chord in the circle. 

We interpret the restricted dominating set problem as a kind of 
\textbf{point covering problem by axis-aligned rectangles} as follows. 
Construct a large square $\mathrm{R}$ with side length $m$. The points 
in $\mathrm{R}$ are indexed by pairs of points on the circle, with $(p_0, p_0)$
and $(p_{m-1}, p_{m-1})$ being, respectively, the lower-left corner of
$\mathrm{R}$ and the upper-right corner of $\mathrm{R}$. For four
points $p_{l_1}, p_{l_2}, p_{r_1}, p_{r_2}$ with $l_1\leq l_2$ and $r_1 \leq r_2$
we denote by $[p_{l_1}, p_{r_1}]\times[p_{l_2}, p_{r_2}] \subseteq \mathrm{R}$ 
the rectangle contained in $\mathrm{R}$ with lower-left point and upper-right
point $(p_{l_1}, p_{r_1})$ and $(p_{l_2}, p_{r_2})$, respectively.

Demand chords are interpreted as \textbf{points in $\mathrm{R}$}. 
The chord $\alpha = (p_l,p_r) \in V^d$ is interpreted as the point 
$Q[\alpha] = (p_l,p_r)$ inside $\mathrm{R}$. Observe that $Q[\alpha]$
is contained above the main diagonal in $\mathrm{R}$, that is the line 
connecting $(p_0, p_0)$ and $(p_{m-1}, p_{m-1})$, as $p_l\leq p_r$.

Covering chords are interpreted as \textbf{pairs of rectangles} contained in $\mathrm{R}$. 
The chord $\beta = (p_l,p_r) \in V^d$ is interpreted as the pair of rectangles
$$
L[\beta] = [p_0, p_l]\times[p_l,p_r] \,\,\, \text{and} \,\,\, 
T[\beta] = [p_l, p_r]\times[p_r,p_{m-1}].
$$
Observe the following property. $L[\beta]$ intersects the left side of $\mathrm{R}$ and 
$T[\beta]$ intersects the top side of $\mathrm{R}$.

It is now straightforward to verify that a covering chord $\beta$
dominates a demand chord $\alpha$ if and only if
$$
Q[\alpha] \in L[\beta] \cup T[\beta].
$$
Finally, we arrived at the desired covering problem, namely 
the problem of selecting a minimum cost set of rectangles pairs 
$L[\beta] \cup T[\beta]$ in $\mathrm{R}$, corresponding to covering chords, 
so as to cover every point $Q[\alpha]$, corresponding to demand chords.
The cost of a rectangle pair is simply the cost of the corresponding
covering chord. Figure~\ref{fig:reduction} illustrates the transformation.

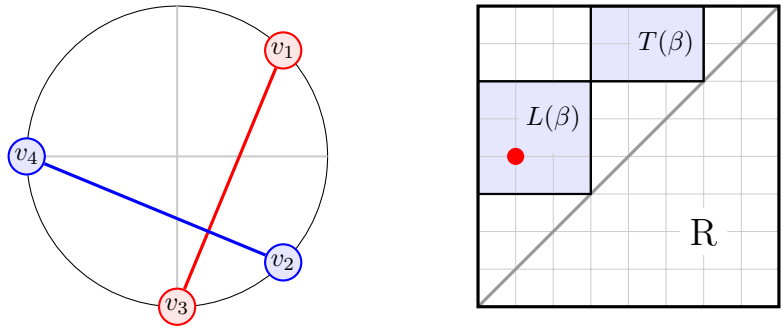
\begin{figure}
\begin{center}
 \begin{tikzpicture}
  \draw (0,0) circle (2cm);
  \draw [draw=black!20, thick] (2,0) -- (-2,0);
  \draw [draw=black!20, thick] (0,2) -- (0,-2);

  \node (a1) at (1.41, 1.41) [circle,draw=red!100,fill=red!10,thick,inner sep=1pt,minimum size=1mm] {\small $v_1$};
  \node (a2) at (1.41, -1.41) [circle,draw=blue!100,fill=blue!10,thick,inner sep=1pt,minimum size=1mm] {\small $v_2$};
  \node (a3) at (0, -2) [circle,draw=red!100,fill=red!10,thick,inner sep=1pt,minimum size=1mm] {\small $v_3$};
  \node (a4) at (-2, 0) [circle,draw=blue!100,fill=blue!10,thick,inner sep=1pt,minimum size=1mm] {\small $v_4$};
  
  \draw [-,red,very thick] (a1) to node [black] {} (a3);
  \draw [-,blue,very thick] (a2) to node [black] {} (a4);

  \begin{scope}[xshift=6cm]
   
   \filldraw [black, thick,fill=blue!10] (-2,-0.5) -- (-2,1) -- (-0.5,1) -- (-0.5,-0.5) -- cycle;
   \filldraw [black, thick,fill=blue!10] (-0.5,1) -- (-0.5,2) -- (1,2) -- (1,1) -- cycle;

    \draw [step=0.5cm,black!20,very thin] (-2,-2) grid (2,2);
   \draw [black!40, very thick] (-2,-2) -- (2,2);
   \draw [black, very thick] (-2,-2) -- (-2,2) -- (2,2) -- (2,-2) -- cycle;

    \draw [black, thick] (-2,-0.5) -- (-2,1) -- (-0.5,1) -- (-0.5,-0.5) -- cycle;
   \draw [black, thick] (-0.5,1) -- (-0.5,2) -- (1,2) -- (1,1) -- cycle;

    \node (p) at (-1.5, 0) [circle,draw=red!100,fill=red!100,thick,inner sep=1pt,minimum size=2mm] {};

    \node (p) at (-1, 0.5) [circle,draw=blue!10,fill=blue!10,thick,inner sep=1pt,minimum size=2mm] {\small $L(\beta)$};
    \node (p) at (0.5, 1.5) [circle,draw=blue!10,fill=blue!10,thick,inner sep=1pt,minimum size=2mm] {\small $T(\beta)$};
    \node (p) at (1,-1) [circle,draw=red!0,fill=red!0,thick,inner sep=1pt,minimum size=2mm] {\Large $\mathrm{R}$};
  \end{scope}

 \end{tikzpicture}
\end{center}
 \caption{An illustration of the transformation. The chords $\alpha = (v_1, v_3)$ and $\beta = (v_2, v_4)$ are
demand and covering chords, respectively. The ordering of the points on the circle is clockwise
starting from the highest point.}\label{fig:reduction}
\end{figure}

The standard LP relaxation for this covering problem reads
\begin{equation*}
\min \left\{ g(z) \,\mid\, z_\beta\ \geq 0 \;\,\, \forall \beta \in V^c, \quad 
\sum\limits_{\beta \,:\, Q[\alpha] \in L[\beta] \cup T[\beta]} z_\beta \geq 1 
\quad \forall \alpha\in V^d \right\}.
\end{equation*}
Let $z$ be a fractional feasible solution to the latter LP. We construct
an integral solution as follows. First, observe that for every demand
chord $\alpha \in V^d$, at least one of the following holds due to
feasibility of $z$.
\begin{itemize}
 \item $\sum_{\beta \,:\, Q[\alpha] \in L[\beta]} z_\beta \geq \frac{1}{2}$
 \item $\sum_{\beta \,:\, Q[\alpha] \in T[\beta]} z_\beta \geq \frac{1}{2}$
\end{itemize}
Let $V^d_L \subseteq V^d$ be the set of all $\alpha \in V^d$ for which
the first condition holds. Let $V^d_T = V^d \setminus V^d_L$ be all other demand
chords. We show how to construct an integral solution of cost at most $4g(z)$ 
that dominates all chords in $V^d_L$. From symmetry, this implies that 
another integral solution can be constructed for $V^d_T$ with cost at most $4g(z)$.
This will then prove the claim, as the union of both solutions is
an integral feasible solution of cost at most $8g(z)$.

To this end observe that $2z$ is a fractional feasible solution to the LP
\begin{equation*}
\min \left\{ g(z) \,\mid\, z_\beta\ \geq 0 \;\,\, \forall \beta \in V^c, \quad 
\sum\limits_{\beta \,:\, Q[\alpha] \in L[\beta]} z_\beta \geq 1 
\quad \forall \alpha\in V^d_L \right\}.
\end{equation*}
Now, it remains to observe that the latter LP is the natural LP relaxation
of an ordinary rectangle covering problem. The rectangles 
$\{L[\beta]\,\mid\, \beta\in V^c\}$ also have the additional property that their
left side lies on the left side of $\mathrm{R}$. This restricted variant 
of the rectangle covering problem was studied by Bansal and Pruhs~\cite{bansal2010geometry},
who proved that the standard LP relaxation of the problem has integrality gap 
of $2$. This implies that there exists an integral solution covering $V^d_L$ 
with cost $4g(z)$. This solution can also be constructed in polynomial time.
This concludes the proof of the claim. 
\end{proof}

\end{proof}

\subsubsection{Putting it all together}

We are ready to prove Theorem~\ref{thm:planar}.

\begin{proof}[Theorem~\ref{thm:planar}]
 The feasibility of the solution obtained after the final augmentation step is
obvious. It remains to compute the approximation guarantee. 
Let $\mathrm{ALG}$ denote the cost of the solution returned by the algorithm. 
Clearly, $\mathrm{AUG}_i\leq 2\mathrm{OPT}$ holds for every $i\in [k]$, as any optimal
solution is feasible for any augmentation problem, and Lemma~\ref{lem:paths} asserts that 
by using unions of paths we lose a factor of at most $2$.
According to Lemma~\ref{lem:rounding},
an $8i$-approximation can be obtained for the $i$-th augmentation problem 
in polynomial time. 
Also, the shortest path comprising the solution of $\mathrm{P}_0$ has cost of at most $\mathrm{OPT}$.
In total, we obtain the bound
$
\mathrm{ALG} \leq \mathrm{OPT} + \sum_{i=0}^k 8i \cdot 2 \mathrm{OPT} = O(k^2) \mathrm{OPT}. 
$

\end{proof}

\section{Bulk-robust spanning trees and further extensions}\label{sec:extensions}

\subsection{Bulk-robust spanning trees}
Let us discuss first the minor changes needed to prove Theorem~\ref{thm:planar}
for the bulk-robust spanning tree problem. 
As the changes
are minor, we choose to follow the outline of the proof given in the main text, and
describe the required modifications.

\subsubsection{The augmentation framework}

We use the same sets $\Omega_i$, $i\in [k]_0$ to define the relaxations
of the problem. The $i$-th augmentation problem $\mathrm{P}_i$ is to augment 
the set $X_{i-1}$ of edges chosen so far to a set $X_i$ with the property that
$(V, X_i\setminus F)$ is a connected graph for all $F\in \Omega_i$. 

As for the bulk-robust $s$-$t$ connection problem, the optimal solution to
any augmentation problem is a forest.
As a consequence, Lemma~\ref{lem:paths} still applies, so we can again use
unions of paths to approximate the augmentation problem, at the loss of a 
factor $2$.

The notion of links and covering by links is defined as before, except that
now cuts formed by sets $F\in \Omega_i$ are arbitrary cuts in the graph, and
not just $s$-$t$ cuts.

\subsubsection{Solving the link covering problem}

The approximate solution to the link covering problems are obtained
in essentially the same way for the bulk-robust spanning tree problem,
as for the bulk-robust $s$-$t$ connection problem. The differences are
minor and are explained next. 

The solution for $\mathrm{P}_{0}$ is computed by computing a minimum spanning
tree in the input graph in polynomial time. The cost of this tree
is clearly at most $\mathrm{OPT}$.

As we did before, we distinguish the case $i=1$ from the case $i\geq 2$.
The case $i=1$ is treated as follows.
The link covering problem corresponding to $\mathrm{P}_{1}$ is no longer
equivalent to an interval covering problem, but it can be approximated as follows. 
Recall that the solution obtained before
the first augmentation problem is a spanning tree of the graph. Each edge
in this tree either belong to some failure scenario $F_i$, in which case 
it comprises a failure scenario in $\Omega_1$, or it is not contained in any
failure scenario. In the latter case the edge can be simply contracted, so
we henceforth assume that all edges of the tree form a scenario in $\Omega_1$.

The augmentation problem now becomes a standard \textbf{connectivity augmentation
problem}, where, given a spanning tree $T$ of a graph $G$, the task is to compute
a minimum-cost set of edges, not in the tree, whose addition to the tree will
increase the size of the minimum cut in the resulting graph to two.
Indeed, on the one hand any set $A$ of edges satisfying that the graph 
$(V, T\cup A)$ has no cut of size one is feasible, as the removal of any 
edge of $T$ cannot disconnect this graph. On the other hand, if a set $A$ is 
such that $(V, T\cup A)$ does contain a cut edge, this edge must belong to $T$
(since $T$ is a spanning tree of $G$). Since all edges of $T$ are assumed to
comprise failure scenarios in $\Omega_1$, this means that $A$ is infeasible 
for the augmentation problem.

It remains to note that the latter connectivity augmentation problem can 
be efficiently approximated within a constant factor. One way to achieve this is
to use the algorithm for survivable network design in~\cite{Jain}.

Consider next the case $i\geq 2$. 
As before, we omit the index $i$ from our notation, as we now discuss the
$i$-th link covering problem for some arbitrary $i\geq 2$.

The set cover LP appropriate for modeling the link covering
problem for the bulk-robust spanning tree problem remain exactly the 
same as before. There is, however, a slight difference in the
design of the separation oracle for the LP. Concretely, the construction 
of the capacitated graph $H$ remains the same, but now violating sets
correspond to sets of edges in minimum cuts (instead of minimum $s$-$t$ cuts),
if the value of the minimum cut is bellow $i+1$. Since minimum cuts
can be found in polynomial time, this separation procedure is also
polynomial.

Finally, the rounding procedure and its analysis remain unchanged. While 
it may seem that the proof of Lemma~\ref{lem:coverproperties} used
the fact that $(\bar V, X\setminus F)$ contains exactly two connected components
$C^s(F)$ and $C^t(F)$, one containing $s$ and the other containing $t$,
the fact that two specific nodes were separated by the cut was never used.
The only property that is used is that $(\bar V, X\setminus F)$ contains exactly
two connected components.
Here we can simply use instead the fact that $(V, X\setminus F)$ contains exactly
two connected components.

This concludes the description of the required modifications.


\subsection{Further extensions}
Let us conclude by discussing some further extensions and implications 
of our techniques. First, our techniques can clearly be applied to other 
bulk-robust network design problems. A treatment of the bulk-robust survivable 
network design problem is deferred to the full version of the paper. 

Also, as we mentioned in the introduction, our methods seem to be suitable
for solving other robust problems in planar graphs. Consider, for example,
the uniform model with varying interdiction costs, where each edge has
an interdiction cost $c(e)\in \mathbb{Z}_{\geq 0}$, and the set of scenarios
is exactly the set of all edge subsets with total interdiction cost at most 
$B\in \mathbb{Z}_{\geq 0}$. Our methods can be used to approximate this
problem provided that a suitable (approximate) separation oracle is provided
for the resulting LP. In general, however, this separation problem coincides
with difficult interdiction problems 
(see e.g.~\cite{israeli_2002_shortestpath,zenklusen_2010_matching} and references therein).



\appendix

\section{Proof of Theorem~\ref{thm:complexity}}\label{apx:complexity}

Recall that a \textbf{hypergraph} is a pair $\HH = (\VV,\EE)$, where $\VV$ is a finite 
set of \textbf{nodes}, and $\EE \subseteq 2^\VV$ is a set of subsets of $\VV$ called \textbf{edges}.
Let $m\in \mathbb{Z}_{\geq 0}$. We say that $\HH$ is \textbf{$m$-uniform} if $|e| = m$ for every 
$e\in \EE$. Observe that $2$-uniform hypergraphs are graphs. 
$\HH$ is \textbf{$m$-partite} if $\VV$ can be partitioned
into $m$ parts $\VV = \VV_1 \cup \cdots \cup \VV_m$ such that for every $e\in \EE$ and for 
every $j\in [m]$ it holds that
$$
|e \cap \VV_j| \leq 1.
$$
A \textbf{vertex cover} of $\HH$
is a set $S\subseteq \VV$ of nodes that touches every edge, i.e. such that $|S\cap e| \geq 1$ holds
for every $e\in \EE$. The \textbf{hypergraph minimum vertex cover} problem is to find
a vertex cover of $\HH$ of minimum cardinality.

Our reduction relies on the following hardness-of-approximation result of 
Guruswami, Sachdeva and Saket~\cite{guruswami2015inapproximability}.

\begin{theorem}[Guruswami et. al.~\cite{guruswami2015inapproximability}]\label{thm:hard}
 For any $\epsilon > 0$ and any $m\geq 4$ it is NP-hard to approximate the 
minimum hypergraph vertex cover problem within a factor $\frac{m}{2} - 1 + \frac{1}{2m} - \epsilon$,
even when the hypergraph is restricted to be $m$-uniform and $m$-partite, and the $m$-partition
is given as input.
\end{theorem}

We show that the minimum hypergraph vertex cover problem on $k$-uniform and $k$-partite
hypergraphs can be transformed to an equivalent instance of bulk-robust $s$-$t$ connection
with diameter $k$, provided that the $k$-partition is given as input.

To this end let $\HH = (\VV,\EE)$ be a $k$-uniform, $k$-partite hypergraph, and 
let $\VV = \VV_1 \cup \cdots \cup \VV_k$ be the given $k$-partition of $\VV$.
We construct the series-parallel graph to be the input of the bulk-robust $s$-$t$
connection problem as follows.

Let $p = |\EE|$ and $n_j = |\VV_j|$ for $j\in [k]$.
For every $j\in [k]$ we construct an ordering $e^j_1, \cdots, e^j_p$ of $\EE$
corresponding to $\VV_j$. This ordering is constructed as follows. First,
order the nodes in $\VV_j$ in an arbitrary way, say $v^j_1, \cdots, v^j_{n_j}$.
Now, construct the ordering of edges by first including all edges incident
$v^j_{1}$ in any order, then all edges incident to $v^j_2$ in any order and so
on, until all vertices are traversed. Since $\VV_j$ is a part in a $k$-partition,
the latter procedure succeeds in producing an ordering of $\EE$, as every edge 
is incident to exactly one node in $\VV_j$. By design, the latter construction
satisfies the following useful property that we will use later. For every node 
$v\in \VV$ there exists an index $j\in [k]$ such that the set of edges
$\EE_v = \{e\in \EE \,\mid\, v\in e\}$ incident to $v$ appear as a sub-sequence in the
$j$-th ordering. This index $j$ can be chosen such that $v\in \VV_j$.

Next, start constructing the series-parallel graph $G = (V,E)$. First, include the
nodes $s,t$ and connect them by $k$ node-disjoint paths $P^1, \cdots, P^k$ 
(the nodes $s$ and $t$ are common to all paths). 
Each path $P^j$ contains exactly $p$ edges $f^j_1, \cdots, f^j_p$, appearing
in this order when $P^j$ is traversed from $s$ to $t$. The edge $f^j_l$ is associated
with the edge $e\in \EE$ of the hypergraph in the $l$-th position of the $j$-th ordering,
i.e., the edge $e^j_l$. Clearly, every edge $e\in \EE$ is associated with exactly 
one edge of $G$ on every path $P^j$ and thus, in total, it is associated with $k$
edges of $G$.

Next, for every $j\in [k]$ and $v\in \VV_j$ add the edge $\alpha_v$ to $G$, connecting
two nodes $w^j_v$ and $z^j_v$ on $P^j$. These nodes are selected so that the set of
edges between these nodes on the path $P^j$ are exactly those that are associated with
edges in $\EE_v$. Since the hypergraph edge in $\EE_v$ appear as a sub-sequence in the
order used to construct $P^j$, such two nodes $w^j_v$ and $z^j_v$ exist.
It is straightforward to verify that $G$ is series-parallel.

To complete the construction of the graph we set the weights of all edges on the paths $P^j$ for
$j\in [k]$ to zero, while the weight of edges of the type $\alpha_v$ for $v\in \VV$
is set to one.

To conclude the reduction it remains to specify the scenario set $\Omega$ of the
bulk-robust $s$-$t$ connection problem. For $e\in \EE$ we include in $\Omega$ a single failure
scenario $F_e \subseteq E$. The set $F_e$ contains the $k$ edges, one from every 
path $P^j$, $j\in [k]$, that are associated with $e$. Since every edge $e\in \EE$
is associated with exactly one edge on every such path, we have $|F_e| = k$, as required.

We conclude the proof by showing that the resulting instance of the bulk-robust $s$-$t$
connection problem is equivalent to the hypergraph vertex cover instance. Formally, 
we show that a solution to the hypergraph vertex cover instance can 
be transformed to a solution of the bulk-robust $s$-$t$ connection instance with 
the same cost, and vice versa.

Assume first that $S\subseteq \VV$ is solution to the hypergraph vertex cover problem. Construct
a solution $X \subseteq E$ to the the bulk-robust $s$-$t$ connection problem as follows.
Include in $X$ all paths $P^j$ for $j\in [k]$ (at zero cost), as well as all edges $\alpha_v$
such that $v\in S$. This solution has cost $|S|$, as required.
To see that this solution is feasible, consider any $F = F_e \in \Omega$.
Since $S$ is a vertex cover, there exists some $v\in S$ such that $v\in e$. Let $j\in [k]$
be such that $v\in \VV_j$. Observe that the path starting at $s$, following $P^j$ until $w^j_v$,
then crossing $\alpha_v$ and then continuing to $t$ on $P^j$ is contained in $X\setminus F$.

Assume next that $X\subseteq E$ is a feasible solution to the bulk-robust $s$-$t$
connection instance. Let $S = \{v\in \VV \,\mid\, \alpha_v\in X\}$. By the cost
structure of the reduction we know that the cost of $X$ is exactly $|S|$. It remains
to prove that $S$ is a vertex cover. Consider any $e\in \EE$. Since $X$ is feasible,
there exists an $s$-$t$ path $Q\subseteq X\setminus F_e$. This $s$-$t$ path must use
some edge of the form $\alpha_v$ for $v\in S$ as, by construction, every path $P^j$ 
intersects every failure scenario, and these are node-disjoint $s$-$t$ paths. 
Furthermore, we can assume that $\alpha_v$ is the only
such edge, as from every path $P^j$, only one edge is contained in $F_e$.
Now, this edge $\alpha_v\in Q$ connects some nodes $w^j_v$ and $z^j_v$ on some path $P^j$.
Consequently, the unique edge in $F_e\cup P^j$ is contained on the sub-path connecting
$w^j_v$ and $z^j_v$ and hence, by construction, $v\in e$. We conclude that $S$ is 
a vertex cover, as required.

The proof of the theorem now directly follows from the latter reduction and 
Theorem~\ref{thm:hard}. \qed



\end{document}